\documentclass[11pt]{amsart}
    \usepackage[a4paper,margin=25mm]{geometry}
    \usepackage{amsmath,amssymb,amsthm}
    \usepackage{parskip}
    \usepackage{needspace}
    \usepackage{booktabs}
    \newtheorem{theorem}{Theorem}
    \newtheorem{lemma}{Lemma}
    \newtheorem{proposition}[lemma]{Proposition}
    \usepackage[T1]{fontenc}
    \usepackage{lmodern}
    \usepackage{microtype}
    \usepackage[hidelinks]{hyperref}
    \title{Amicable numbers} 
    \author{Paul Pollack}
    \address{Department of Mathematics, University of Georgia, Athens, GA 30602, USA}
    \subjclass[2020]{Primary 11A25; Secondary 11N37}
    \keywords{Amicable numbers, sum-of-divisors function, aliquot sums}
    
\begin{document}
    \begin{abstract} Let $s(n):=\sum_{d\mid n,\,d < n} d$ denote the sum of the proper divisors of $n$. Distinct positive integers $n,m$ form an \textsf{amicable pair} if  $s(n)=m$ and $s(m)=n$. 
Let $A(x)$ count the positive integers not exceeding $x$
that belong to an amicable pair.
We report on a proof, found by ChatGPT Astra, that
$$ A(x)\le x\exp\{-(1/2+o(1))\log x\,\log_3 x/\log_2 x\}$$
as $x\to\infty$. Here $\log_k$ denotes the $k$th iterate of the natural logarithm.
\end{abstract}
    \maketitle
    
    \section{Introduction}\label{sec:introduction}

Let $\sigma(n):=\sum_{d\mid n}d$ denote the sum of the positive divisors of $n$ and write $s(n):=\sum_{d\mid n,\,d<n}d$ for the sum of the proper divisors of $n$, so that $s(n)=\sigma(n)-n$. There are many problems about $s(n)$ which go back to ancient times and about which we still know embarrassingly little. Perhaps the most famous of these concern \textsf{perfect numbers}, positive integers $n$ for which $s(n)=n$. In this article we are concerned with a close conceptual cousin of the perfect numbers: amicable pairs. These are pairs of distinct positive integers \(n,n'\) satisfying \(s(n)=n'\) and \(s(n')=n\). We say $n$ is an \textsf{amicable number} if $n$ is a member of an amicable pair. 

The smallest amicable pair $220, 284$ appears in the writings of Iamblichus (ca.~300 CE), who discusses its significance to the Pythagorean school \cite[pp.~34--35]{Iamblichus1894}:\footnote{Translation by ChatGPT from the Greek text.}
\begin{quote}
    For they call certain other numbers, quite explicitly, friends, in assigning virtues and noble dispositions to numbers --- for example, 284 and 220. For the parts of each of these are productive of the other, in accordance with the principle of friendship, as Pythagoras declared.

For when someone asked him, “What is a friend?”, he replied, “Another I”; and this is demonstrated in the case of these numbers.
\end{quote}

Although amicable numbers have been of interest for millennia, the study of their distribution began much more recently, in the 1950s. Let $A(x)$ denote the number of amicable numbers not exceeding $x$. While over a billion amicable pairs are known \cite{ChernykhDatabase}, we do not know how to prove that there are infinitely many. That is, we cannot prove that $A(x)\to\infty$ as $x\to\infty$. We have had much more luck with upper bounds for $A(x)$. Table~\ref{tab:amicable-bounds} records our progress in this direction. Here and below, $\log_{k} x$ denotes the $k$th iterate of the natural logarithm.

\begin{table}[!ht]
\centering
\caption{Successive upper bounds for $A(x)$.}
\label{tab:amicable-bounds}
\small
\renewcommand{\arraystretch}{1.65}
\setlength{\tabcolsep}{4pt}
\begin{tabular*}{\textwidth}{@{\extracolsep{\fill}}llll@{}}
\toprule
Result & Author(s) & Year & Reference\\
\midrule
$\limsup_{x\to\infty} A(x)/x<0.204$
 & Kanold & 1954 & \cite{Kanold1954}\\
$A(x)=o(x)$
 & Erd\H{o}s & 1955 & \cite{Erdos1955}\\
$A(x)\ll x\log_5 x/\sqrt{\log_4 x}$
 & Rieger & 1973 & \cite{Rieger1973}\\
$A(x)\ll x/\log_3 x$
 & Erd\H{o}s--Rieger & 1975 & \cite{ErdosRieger1975}\\
$A(x)\ll x\exp\{-c\sqrt{\log_3 x\,\log_4 x}\}$
 & Pomerance & 1977 & \cite{Pomerance1977}\\
$A(x)\ll x\exp\{-c(\log x\,\log_2 x)^{1/3}\}$
 & Pomerance & 1981 & \cite{Pomerance1981}\\
$A(x)\le x\exp\{-(\tfrac12+o(1))\sqrt{\log x\,\log_3 x}\}$
 & Pomerance & 2015 & \cite{Pomerance2015}\\
\bottomrule
\end{tabular*}

\smallskip
\begin{minipage}{\textwidth}
\footnotesize
Here $c>0$ is an absolute constant, which may differ across rows.
\end{minipage}
\end{table}

In this paper, we establish the following new upper bound, whose proof was discovered by ChatGPT Astra. In what follows,
\[ L(x) := \exp\left(\log{x} \frac{\log_3{x}}{\log_2{x}}\right).\]

    \begin{theorem}\label{thm:main}
    As $x\to\infty$, we have
    \[
     A(x)\le x L(x)^{-\frac{1}{2}+o(1)}.
    \]
    \end{theorem}
We expect that this is the first upper bound on $A(x)$ of the
``correct shape.'' Let $N_{\phi}(n)$ and $N_{\sigma}(n)$ denote
the numbers of solutions $m$ to the equations $\phi(m)=n$ and
$\sigma(m)=n$, respectively. Under plausible (still unproved)
hypotheses on the distribution of smooth shifted primes $p-1$,
Pomerance \cite{Pomerance1980}---developing ideas of Erd\H{o}s
\cite{Erdos1935}---showed that there are numbers $m\le x$ for which
$N_{\phi}(m)\ge x/L(x)^{1+o(1)}$ as $x\to\infty$. An analogous
argument, using smooth shifted primes $p+1$, suggests that there
are $m\le x$ for which $N_{\sigma}(m)\ge x/L(x)^{1+o(1)}$.
For such an $m$, there are at least $x^2/L(x)^{2+o(1)}$ pairs
$n<n'$ with $\sigma(n)=\sigma(n')=m$. Each sum $n+n'$ lies in
$[2,2m]\subset[2,2x]$. If the particular sum $m$ occurs with
roughly the average frequency, this suggests at least
$x/L(x)^{2+o(1)}$ pairs with $n+n'=m$. Each such pair is
amicable. This heuristic, which in less precise form goes back to Erd\H{o}s, suggests that Theorem~\ref{thm:main} has the correct shape and that the constant \(1/2\) cannot be replaced by anything larger than \(2\).

It is plausible that $2$ is the correct exponent, i.e., that
$A(x)=x/L(x)^{2+o(1)}$. It is proved in
\cite[Corollary 2]{Pollack2019} that
$\sum_{m\le x}N_\phi(m)^2\le x^2/L(x)^{2+o(1)}$.
Using \cite[Corollary 1.2]{Pollack2021}, together with a
decomposition into squarefull and squarefree parts, one can show
the analogous bound
$\sum_{m\le x}N_\sigma(m)^2\le x^2/L(x)^{2+o(1)}$.
Now put $X:=x\log x$. If $n\le x$ belongs to an amicable pair
$n,n'$, then their common divisor sum satisfies
$m:=\sigma(n)=\sigma(n')\le X$ for sufficiently large $x$.
(Here we use the familiar bound
$\sigma(v)\ll v\log_2(3v)$ for $v\ge1$
\cite[Theorem 323, p.~350]{HardyWright2008}.)
The second-moment bound for $N_\sigma$ shows that there are
at most $X^2/L(X)^{2+o(1)}=x^2/L(x)^{2+o(1)}$ pairs
$n<n'$ with $\sigma(n)=\sigma(n')=m \le X$. If the additional condition \(n+n'= m\) holds for a proportion \(\approx x^{-1}\) of these pairs, as the preceding heuristic suggests, this would give \(A(x)\le x/L(x)^{2+o(1)}\).

\subsection*{Definitions and notation} The letters $p,q,\ell$, with primes or subscripts attached, always denote prime numbers. We write $\Omega(n)$ for the number of prime factors of $n$, counted with multiplicity. We let $P^{+}(n)$ denote the largest prime factor of $n$, setting $P^{+}(1):=1$. For $n\ge1$ and $y\ge2$, we let $n_{\le y}$ and $n_{>y}$ denote the \textsf{$y$-smooth} and \textsf{$y$-rough} parts of $n$, respectively. That is,
\[ n_{\le y}:=\prod_{\substack{q^e\parallel n\\q\le y}}q^e, \qquad
 n_{>y}:=\prod_{\substack{q^e\parallel n\\q>y}}q^e,
\]
where empty products are taken to be $1$. Thus $n_{>y}=n/n_{\le y}$. We say that $d$ is a \textsf{unitary divisor} of $n$ if $d\mid n$ and $(d,n/d)=1$.

\subsection*{Statement of AI usage} The proof of Theorem~\ref{thm:main} was produced by ChatGPT Astra.
At this particular moment in time, LLMs such as Astra are
brilliant at solving problems, but their explanations are liable to
cause acute intellectual indigestion. P.P., in extended conversation with
Astra, cleaned up the argument in the hope of making it
palatable to carbon-based lifeforms. P.P. takes full
responsibility for the mathematical content without claiming
credit for discovering the proof.

ChatGPT Astra also produced a Lean formalization of
Theorem~\ref{thm:main}. The formal proof has been checked by
Lean's kernel; the only axioms on which it depends are
\texttt{propext}, \texttt{Classical.choice}, and
\texttt{Quot.sound}, the standard foundational axioms used by
Lean and Mathlib. See
\url{https://github.com/twonth/amicable-numbers-lean}.
    \section{Proof of Theorem \ref{thm:main}}\label{sec:main-proof}
    
    We first isolate two lemmas needed for the proof of Theorem \ref{thm:main}. 
    The first of these, Lemma \ref{lem:pollack}, appears as \cite[Theorem 1.1]{Pollack2011}. 
    
    \begin{lemma}\label{lem:pollack}
    Fix $\beta>0$. There is a constant $\rho=\rho(\beta)>0$ for which the following holds. For all $X>X_0(\beta)$ and all $G>\exp((\log\log X)^\beta)$,
    \[
     \#\{n\le X:\gcd(n,\sigma(n))>G\}\le X/G^\rho.
    \]
    \end{lemma}

\begin{lemma}\label{lem:technical-A}
Fix $\kappa>0$. Let $x$ be large, and set $X:=x\log x$ and $y:=\log x$. Uniformly for $1\le Y\le X$, we have
\[
 \#\{b\le Y:b\text{ squarefree},\ 
 \log(\sigma(b)_{\le y})+(\log_2 x)\Omega(\sigma(b)_{>y})>\kappa\log x\}
 \le YL(x)^{-\kappa+o(1)}.
\]
Here $o(1)$ denotes a quantity tending to $0$ as $x\to\infty$.
\end{lemma}

\begin{proof}
We apply Rankin's trick. Put $H(v):=v_{\le y}y^{\Omega(v_{>y})}$. For any choice of $\eta,\delta>0$, the count of $b$ in question is at most
\begin{align*}
 Y^{1+\delta}x^{-\kappa\eta}
 \sum_{\substack{b\le Y\\b\text{ squarefree}}}
 \frac{H(\sigma(b))^\eta}{b^{1+\delta}}
 &\le Y^{1+\delta}x^{-\kappa\eta}
 \prod_q\left(1+\frac{H(q+1)^\eta}{q^{1+\delta}}\right)\\
 &\le Y^{1+\delta}x^{-\kappa\eta}
 \exp\left\{\sum_q\frac{H(q+1)^\eta}{q^{1+\delta}}\right\}.
\end{align*}
We let
\[
 A:=\frac{\log_2 x}{(\log_3 x)^2},\qquad
 \eta:=\frac{\log A}{\log_2 x},\qquad
 \delta:=\frac1{\log_2 x}.
\]
Then $Y^{\delta}\le X^{\delta}\le L(x)^{o(1)}$ and $x^{-\kappa\eta}=L(x)^{-\kappa+o(1)}$. So it suffices to show that the displayed sum on $q$ is $o(\log L(x))$. Crudely,
\[
 \sum_q\frac{H(q+1)^\eta}{q^{1+\delta}}
 \le 2^{1+\delta}\sum_q\frac{H(q+1)^\eta}{(q+1)^{1+\delta}}
 <3\sum_{v\ge1}\frac{H(v)^\eta}{v^{1+\delta}},
\]
where the final sum is over all positive integers $v$. Continuing, since $y^\eta=A$,
\begin{align*}
 \sum_{v\ge1}\frac{H(v)^\eta}{v^{1+\delta}}
 &=\prod_{p\le y}(1-p^{-1-\delta+\eta})^{-1}
   \prod_{p>y}(1-Ap^{-1-\delta})^{-1}\\
 &\le\exp\left\{\sum_{p\le y}p^{-1-\delta+\eta}
       +A\sum_{p>y}p^{-1-\delta}+O(1)\right\}\\
 &\le\exp\left\{A\sum_{p\le y}p^{-1}
       +A\sum_{p>y}p^{-1-\delta}+O(1)\right\}.
\end{align*}
Here the higher degree terms in the logarithmic expansion have been absorbed into the $O(1)$, using that $\eta,\delta=o(1)$ and $A^2/y=o(1)$.

Now, by Mertens' theorem \cite[Theorem 2.7(d), p.~50]{MontgomeryVaughan2007},
\[
 A\sum_{p\le y}p^{-1}\le(1+o(1))A\log_2 y
 =(1+o(1))\log_2 x/\log_3 x,
\]
while the Euler product and standard bounds for $\zeta(s)$ \cite[Corollaries 1.11 and 1.14, pp.~23, 25]{MontgomeryVaughan2007} give
\[
 \sum_{p>y}p^{-1-\delta}\le\sum_p p^{-1-\delta}
 \le\log\zeta(1+\delta)+O(1)
 \le\log\frac1\delta+O(1)=\log_3 x+O(1).
\]
Collecting estimates, we conclude that the above sum on $q$ is at most $\exp\{(2+o(1))\log_2 x/\log_3 x\}\allowbreak=(\log x)^{o(1)}$, which is comfortably $o(\log L(x))$.
\end{proof}

    The next proposition reduces the proof of Theorem \ref{thm:main} to handling pairs whose members have large, relatively prime squarefree divisors with controlled large prime factors. 
    The proof borrows several ideas from the preliminary reduction steps in Pomerance's paper \cite{Pomerance2015}; see the start of \S3 there (compare also \S2 of \cite{Pomerance1981}). 

    \begin{proposition}\label{prop:structural}
    Let $x$ be large. Apart from at most \(xL(x)^{-1/2+o(1)}\) exceptions, every amicable pair \(\{n,n'\}\) with at least one member not exceeding \(x\) admits, after possibly interchanging \(n,n'\), a factorization
    \[
     n=Ba,\qquad n'=B'a',
    \]
    where
    \[
    \begin{gathered}
     B,B'\le L(x)^{O(1)},\qquad a,a'=x/L(x)^{O(1)},\\
     a,a'\text{ squarefree},\\
     \gcd(a,B)=1,\qquad \gcd(a',B')=1,\\
     \gcd(a,n')=\gcd(a',n)=1,\\
     P^{+}(a')<P^{+}(a)\le x^{4/5}.
    \end{gathered}
    \]
    \end{proposition}
   Below, ``negligible exceptional set'' always means a set of size at most $xL(x)^{-1/2+o(1)}$.
    
    \begin{proof}
    Throughout the proof, we assume that $x$ is sufficiently large. We put
    \[
     X:=x\log x,\quad Y:=L(x)^{1/2}.
    \]
    At least one member of the pair $n, n'$ does not exceed $x$. As noted in the introduction, both $n,n'\le X$. Let $\rho:=\rho(1/2)$ be the constant of Lemma~\ref{lem:pollack} corresponding to $\beta=1/2$. Then we may assume that
    \[
     (n,n')\le Y^{1/\rho}.
    \]
    Indeed, $(n,n')=(n,\sigma(n)-n)=(n,\sigma(n))$, while Lemma~\ref{lem:pollack} with $G:=Y^{1/\rho}$ bounds the number of $n\le X$ with $(n,n')>Y^{1/\rho}$ by $XY^{-1}$. This is $xL(x)^{-1/2+o(1)}$ and so negligible for us.
    
    We may also assume that $n,n'>XY^{-1}$ and that $n,n'$ both have squarefull part at most $Y^2$, since the count of exceptions in $[1,X]$ to either condition is $O(XY^{-1})$. (We use that the count of squarefull numbers up to $t$ is $O(t^{1/2})$ \cite[p.~101, (14)]{ErdosSzekeres1934}, then apply partial summation.) We have noted here that each member of an amicable pair determines the pair.
    
    Let $B$ be the product of all prime powers $q^e\parallel n$ with $e\ge 2$, along with all primes $q\parallel n$ that divide $n'$. Define $B'$ analogously. Then
    \[
     B,B'\le Y^2\cdot Y^{1/\rho}=L(x)^{O(1)}.
    \]
    Writing $n=Ba$ and $n'=B'a'$, we see that $a,a'$ are squarefree, $(a,B)=(a',B')=1$, and $(a,n')=(a',n)=1$. Moreover,
    \[
     X\ge a=n/B\ge\frac{X/Y}{L(x)^{O(1)}}=x/L(x)^{O(1)},
    \]
    so that $a=x/L(x)^{O(1)}$, with the same estimate holding for $a'$.
    
    Since $(a,n')=1$, we have $(a,a')=1$, and so $P^{+}(a)$ and $P^{+}(a')$ are distinct. Without loss of generality, we may assume that $P^{+}(a')<P^{+}(a)$. It remains to show that $P^{+}(a)\le x^{4/5}$, apart from a negligible exceptional set.
    
    Write $n=mp$ and $n'=m'p'$, where $p:=P^{+}(a)$ and $p':=P^{+}(a')$. The relations $n'=s(n)$ and $n=s(n')$ give 
    \begin{align*}
     m'p'&=ps(m)+\sigma(m),\\
     mp&=p's(m')+\sigma(m').
    \end{align*}
    We eliminate $p$ from the two equations: multiply the first equation by $m$, the second by $s(m)$, and then substitute the second into the first. We find that
    \[
     p'(mm'-s(m)s(m'))=m\sigma(m)+s(m)\sigma(m')>0,
    \]
    so that $m,m'$ determine $p'$, and hence also determine $n=pm$. Hence, the number of pairs where $mm'\le XY^{-1}$ is $O(XY^{-1}\log x)$, which is negligible for us. So we assume in what follows that $mm'>XY^{-1}$.
    
Suppose that $p>x^{4/5}$. Then
    \[
     m\le X/p\le x^{1/5+o(1)}.
    \]
    Furthermore, as $mm'>XY^{-1}$, we have that $m'\ge XY^{-1}m^{-1}\ge x^{4/5+o(1)}$, so that
    \[
     p'\le x^{1/5+o(1)}.
    \]
    Let $D$ be the smallest divisor of $a'$ exceeding $x^{1/2}$. Then
    \[
     x^{1/2}<D\le x^{7/10+o(1)}.
    \]
    Write $n'=DM$. Then $(D,M)=1$. Furthermore,
    \[
     (D,\sigma(D))\mid(D,\sigma(n'))=(D,n+n')=(D,n)\mid(a',n)=1,
    \]
    so that $(D,\sigma(D))=1$. We count the possible $n$ for fixed $D$ and $m$, then sum over $D,m$.
    
    The relations $n'=s(n)$ and $\sigma(n')=n+n'$ give
    \begin{align*}
     DM&=ps(m)+\sigma(m),\\
     \sigma(D)\sigma(M)&=pm+DM.
    \end{align*}
    Eliminating $p$, we find that
    \[
     s(m)\sigma(D)\sigma(M)=\sigma(m)DM-m\sigma(m).
    \]
    Looking modulo $\sigma(D)$,
    \[
     \sigma(m)DM\equiv m\sigma(m)\pmod{\sigma(D)}.
    \] 
    Since $(D,\sigma(D))=1$, the last congruence places $\sigma(m)M$ in a unique residue class modulo $\sigma(D)$, determined by $D,m$. Hence, if a solution exists, $M$ lies in a single residue class modulo $\sigma(D)/(\sigma(D),\sigma(m))$. Since $M\le X/D$, the number of possibilities for $n'=DM$ given $m,D$ is at most
    \[
     1+\frac{X/D}{\sigma(D)/(\sigma(D),\sigma(m))}
     \le1+\frac{X\sigma(m)}{D^2}.
    \]
    For large $x$, we have both $m\le x^{0.21}$ and $D\le x^{0.71}$ (say). Consequently, the number of $n$ that arise this way is at most
    \[
    \begin{aligned}
     \sum_{m\le x^{0.21}}\sum_{x^{1/2}<D\le x^{0.71}}
     \left(1+\frac{X\sigma(m)}{D^2}\right)
     &\ll x^{0.92}+X\cdot(x^{0.21})^2x^{-1/2}\\
     &\ll x^{0.92}\log x,
    \end{aligned}
    \]
    and this is again negligible. (We use here that
    $\sum_{m\le z}\sigma(m)\ll z^2$; see \cite[Theorem 324, p.~351]{HardyWright2008}.) \end{proof}

\Needspace{7\baselineskip}
We may now complete the proof of Theorem \ref{thm:main}.

\begin{proof}[Proof of Theorem \ref{thm:main}]
We continue to take $X:=x\log x$, so that both $n,n'\le X$. We discard the exceptional set of Proposition \ref{prop:structural} and assume $n=Ba$ and $n'=B'a'$ with $a,a', B, B'$ as described there. Fix $0<\varepsilon<1/10$. We will prove that
\[
 A(x)\le xL(x)^{-1/2+4\varepsilon+o(1)}, 
\]
as $x\to\infty$. Since $\varepsilon$ can be taken arbitrarily small, this proves Theorem \ref{thm:main}. 

Throughout this proof, constants and $o(1)$ terms may depend on $\varepsilon$.

Here is a brief roadmap of the proof. In each of the two cases below, we write one member of the pair as $mR$ and
choose a factor $v$ of the other member.\footnote{For readers
returning to this outline after reading the proof, $v=p$ in
Case I and $v=V$ in Case II.}
The choice is made so that $\sigma(m)$ and $\sigma(v)$ have
a large common divisor $d$. We count by fixing $m$ first.
There are few choices for $d$, and for each such choice,
the divisibility $d\mid\sigma(v)$ and the size restrictions
on $v$ limit the number of possible $v$. Finally, for fixed
$m,v$, the amicable relations allow at most two admissible
values of $R$.

\medskip
\Needspace{5\baselineskip}
\noindent\textbf{Case I: $P^{+}(a)>x^{\varepsilon^2}$.}

Put $p:=P^{+}(a)$. Then $p\parallel n$, and $p+1 \mid \sigma(n) = \sigma(n') = \sigma(B') \sigma(a')$. Hence, 
\[ r:= \frac{p+1}{g} \mid \sigma(a'), \quad\text{where we set}\quad g:=\gcd(p+1,\sigma(B')).\]
Since 
\[ r \mid \sigma(a') = \prod_{q \mid a'} (q+1), \]
we can factor
\[ r = \prod_{q \mid a'} e_q, \quad\text{where each}\quad e_q \mid q+1.\]
(Here we allow the possibility that some $e_q$ are equal to $1$.)

Let $r_0$ denote the product of those $q \mid a'$ for which $e_q \le q^{9/10}$. Then 
\[ r_0 = a' \Bigg(\prod_{\substack{q\mid a' \\ e_q > q^{9/10}}} q\Bigg)^{-1} \ge a' \Bigg(\prod_{q\mid a'} e_q^{10/9}\Bigg)^{-1} = a'/r^{10/9}.\]
Now $a'=x^{1+o(1)}$, while $r\le p+1\le x^{4/5}+1$. Hence,
\[ r_0\ge\frac{a'}{(p+1)^{10/9}}\ge x^{1/9+o(1)}>x^{1/10}.\]

Put $\alpha:=\varepsilon^2/10$. If there is a prime divisor of $r_0$ exceeding $x^\alpha$, we let $R = P^{+}(r_0)$. Since $R$ is a prime dividing $a'$, we then have
\[ x^{\alpha} < R < p.
\]
Otherwise, we let $R$ be the smallest divisor of $r_0$ exceeding $x^{\alpha}$ (which certainly exists by the above lower bound on $r_0$). In that case, 
\[ x^{\alpha} < R \le x^{2\alpha}.\]
Note that we always have $R \mid r_0 \mid a'\mid n'$, while 
\[ \prod_{q\mid R} e_q \le R^{9/10}.\]

Write $n'=mR$. Then $(m,R)=1$. Set
\[
 d:=\frac{p+1}{\prod_{q\mid R}e_q} = g 
 \prod_{q\mid a'/R} e_q.
\]
Since $g$ divides $\sigma(B')$ and $\prod_{q\mid a'/R} e_q$ divides $\prod_{q\mid a'/R} (q+1) = \sigma(a'/R)$,
we have that
\[ d \mid \sigma(B') \sigma(a'/R) = \sigma(n'/R) = \sigma(m). \]
Moreover,
\[ d > p/R^{9/10}. \]
We count possible values of $n'=mR$ by first fixing $m$, then counting corresponding values of $d,p$, and $R$. 

We say that \(R\) is of Type A if \(R\) is prime and \(R<p\), and of Type B if \(x^\alpha<R\le x^{2\alpha}\). Note that $R$ may be of both types.

For fixed $m,p$, there is at most one admissible $R$ of Type A. Indeed, $p\mid n=s(mR)$ gives \begin{equation}\label{eq:Rsmcong} Rs(m)+\sigma(m)\equiv0\pmod p.\end{equation} If an admissible solution exists, $(p,s(m))=1$; otherwise $p\mid\sigma(m)$ and hence $p\mid \sigma(m)-s(m)=m$, contrary to $(a,n')=1$. The congruence \eqref{eq:Rsmcong} then places the prime $R<p$ into a single residue class modulo $p$, determined by $m,p$.

There is also at most one admissible $R$ of Type B. We have $(R,\sigma(R))=1$. Indeed, 
\[ (R,\sigma(R)) \mid \sigma(m)\sigma(R)-mR = n, \quad (R,\sigma(R))\mid R \mid a', \]
and $\gcd(a',n)=1$. Moreover, $(p,mR)=1$, using that $a$ and $n'$ are relatively prime. The condition $p\mid n$ can be written as the congruence
\[
 \sigma(m)\sigma(R)\equiv mR\pmod p.
\]
If this congruence has an admissible solution, then $(p,\sigma(m))=1$, allowing us to rewrite it in the suggestive form
\[ \frac{\sigma(R)}{R} \equiv \frac{m}{\sigma(m)}\pmod{p}.\]
So if $R_1,R_2$ are two admissible solutions, then $\frac{\sigma(R_1)}{R_1} \equiv \frac{\sigma(R_2)}{R_2} \pmod{p}$, implying that 
\[
 \sigma(R_1)R_2-\sigma(R_2)R_1\equiv0\pmod p.
\]
But if $R_1, R_2$ are of Type B, the absolute value of the left-hand side is at most $x^{4\alpha+o(1)}<p$. Hence it vanishes, giving $\sigma(R_1)/R_1=\sigma(R_2)/R_2$. Since both fractions are reduced, $R_1=R_2$. (A similar idea appears in the proof of \cite[Theorem 3.3]{PollackSinghaRoy2022}.)

To finish up, we place $p,R$ in dyadic ranges $P\le p<2P$, $U\le R<2U$, where $P,U$ are powers of $2$ with $U > x^{\alpha}/2$. For each $m\le X/U$, there are at most $\tau(\sigma(m))=x^{o(1)}$ choices of $d$ \cite[Theorem 315, p.~343]{HardyWright2008}. Each admissible $d$ satisfies $d\ge P/(2U)^{9/10}$, and $p\equiv-1\pmod d$, so there are $O(U^{9/10})$ possible $p$ in the range. As we have just seen, each $m,p$ gives at most two possible $R$. The number of pairs in these ranges is therefore at most
\[
 \frac XU x^{o(1)}U^{9/10}
 \le x^{1-\alpha/10+o(1)}.
\]
There are $O((\log x)^2)$ ranges, so Case I saves a power of $x$.

\medskip
\Needspace{5\baselineskip}
\noindent\textbf{Case II: $P^{+}(a)\le x^{\varepsilon^2}$.}

Since $a$ is of size $x^{1+o(1)}$, the smallest divisor $R$ of $a$ with $R\ge x^{\varepsilon}$ in fact satisfies
\[
 x^\varepsilon\le R\le x^{\varepsilon+\varepsilon^2}.
\]
We write $a=Rb$, so that
\[
 n=RBb.
\]
By Lemma \ref{lem:technical-A}, we can assume that
\begin{equation}\label{eq:case-II-mass}
 \log(\sigma(b)_{\le y})+(\log_2 x)\Omega(\sigma(b)_{>y})
 \le (1/2-4\varepsilon)\log x,
 \quad \text{where~} y:=\log x.
\end{equation}
Indeed, given $b$ there are at most $X/b$ possibilities for $n$. Lemma \ref{lem:technical-A} with $\kappa:=1/2-4\varepsilon$ and partial summation show that the number of exceptions is at most
\[
 (X\log X)\,L(x)^{-1/2+4\varepsilon+o(1)}
 =xL(x)^{-1/2+4\varepsilon+o(1)}.
\]

Since $b$ is a unitary divisor of $n$, we have that $\sigma(b)\mid \sigma(n) = \sigma(n') = \sigma(a')\sigma(B')$. So if we set
\[
 r:=\frac{\sigma(b)_{>y}}{g}, \quad\text{where now}\quad g:= (\sigma(b)_{>y}, \sigma(B')),
\]
then $r$ divides both $\sigma(b)$ and $\sigma(a')$.

Our plan is to choose divisors $V\mid a'$ and $d\mid r$ with
$d\mid\sigma(V)$. Writing $m:=Bb$ (so that $n=mR$), we then have
$d\mid\sigma(m)$, so there are few choices for $d$ once
$m$ is fixed. The estimate \eqref{eq:reciprocal-b} below will allow us to count \(V\), provided \((V/d)(\log x)^{\Omega(d)}\) satisfies a suitable upper bound. We will also show that fixed $m,V$ determine at most one
admissible $R$. We proceed now to construct the required $V,d$.

We can factor
\[
 a'=V_0V_1\cdots V_t,\qquad
 V_0\le x^{2\varepsilon+3\varepsilon^2},\qquad
 x^{2\varepsilon+3\varepsilon^2}<V_i
 \le x^{2\varepsilon+4\varepsilon^2}\quad(1\le i\le t).
\]
To determine $V_1,\dots,V_t$, we successively multiply primes dividing $a'$ until their product first exceeds $x^{2\varepsilon+3\varepsilon^2}$, and repeat with the remaining primes as long as possible. We let $V_0:=a'/(V_1\cdots V_t)$. The upper bound for each $V_i$ follows from $P^+(a')<P^+(a)\le x^{\varepsilon^2}$.
Since $r\mid\sigma(a')$, there is a decomposition
\[
 r=d_0\cdots d_t, \quad\text{where each}\quad d_i\mid\sigma(V_i).
\]

With some malice aforethought, we look for an index $i\ge 1$ for which $(V_i/d_i) (\log{x})^{\Omega(d_i)}$ is small. Such an $i$ can be obtained by averaging. Notice that
\[
 \sum_{i=0}^t
 \{\log(V_i/d_i)+(\log_2 x)\Omega(d_i)\}
 =\log(a'/r)+(\log_2 x)\Omega(r).
\]
By the definition of $r$,
\[
\begin{aligned}
 \log(a'/r)
 &=\log\frac{a'}{\sigma(b)}
   +\log(\sigma(b)_{\le y})
   +\log{g}\\
 &\le\log(a'/b)+\log(\sigma(b)_{\le y})
   +\log\sigma(B').
\end{aligned}
\]
Also, $\Omega(r)\le\Omega(\sigma(b)_{>y})$. Thus
\eqref{eq:case-II-mass} gives
\[
 \log(a'/r)+(\log_2 x)\Omega(r)
 \le\log(a'/b)+\log\sigma(B')
    +(1/2-4\varepsilon)\log x.
\]
Now $\log\sigma(B')=o(\log x)$.  Furthermore, writing $b=a/R$, we see that
\[ \log(a'/b)=\log(a'/a)+\log R
 \le(\varepsilon+\varepsilon^2+o(1))\log x.
\]
(We use here that $a, a'$ are both of size $x^{1+o(1)}$, so that $a'/a= x^{o(1)}$.) 
Consequently,
\begin{equation}\label{eq:total-cost}
 \sum_{i=0}^t
 \{\log(V_i/d_i)+(\log_2 x)\Omega(d_i)\}
 \le(1/2-3\varepsilon+\varepsilon^2+o(1))\log x.
\end{equation}

Inequality \eqref{eq:total-cost} forces the existence of $i\ge1$ for which
\begin{equation}\label{eq:choice-dV}
 \log(V_i/d_i)+(\log_2 x)\Omega(d_i)
 \le(1/2-\varepsilon)\log V_i.
\end{equation}
Otherwise,
\begin{equation}
\begin{aligned}
 \sum_{i=1}^t
 \{\log(V_i/d_i)+(\log_2 x)\Omega(d_i)\}
 &>(1/2-\varepsilon)\sum_{i=1}^t\log V_i\\
 &=(1/2-\varepsilon)\log(a'/V_0)\\
 &\ge\{(1/2-\varepsilon)(1-2\varepsilon-3\varepsilon^2)
       +o(1)\}\log x.\label{eq:contradict1}
\end{aligned}
\end{equation}
On the other hand, the term for $i=0$ in \eqref{eq:total-cost}
is at least $-\log(\sigma(V_0)/V_0)$, which is at least $-\log_2{x}$ for large $x$. So we get from \eqref{eq:total-cost} that
\begin{equation}\label{eq:contradict2}
 \sum_{i=1}^t
 \{\log(V_i/d_i)+(\log_2 x)\Omega(d_i)\}
 \le(1/2-3\varepsilon+\varepsilon^2+o(1))\log x.
\end{equation}
But \eqref{eq:contradict1} and \eqref{eq:contradict2} are incompatible for large $x$, since
\[
 (1/2-\varepsilon)(1-2\varepsilon-3\varepsilon^2)
 >1/2-3\varepsilon+\varepsilon^2.
\]

Choose $i$ satisfying \eqref{eq:choice-dV} and put $d:=d_i$, $V:=V_i$, and $J:=\Omega(d)$.
Then $d>1$, $d\mid\sigma(b)$, $d\mid\sigma(V)$. Furthermore, \eqref{eq:choice-dV} gives
\begin{equation}\label{eq:dVstuff}
 \frac Vd \exp(J\log_2{x})\le x^{\varepsilon-4\varepsilon^3}.
\end{equation}
Here we used that $V \le x^{2\varepsilon + 4\varepsilon^2}$, while 
$(1/2-\varepsilon)(2\varepsilon+4\varepsilon^2)
=\varepsilon-4\varepsilon^3$.

Writing \(n=mR\), we have \(n'=s(n)=\sigma(m)\sigma(R)-mR\). As \(V\mid n'\),
\begin{equation}\label{eq:congruence-R}
 \sigma(m)\sigma(R)\equiv mR\pmod V.
\end{equation}

We first show that fixed $m,V$ determine at most one admissible
value of $R$, and hence at most one possibility for $n=mR$.
To see why, first observe that $(R,\sigma(R))=1$. This coprimality is proved exactly as in Case I: We have that $R\mid a$ and that $(R,\sigma(R)) \mid \sigma(m)\sigma(R)-mR = n'$, while $\gcd(a,n')=1$.
We also have that $(R,V)=1$: indeed, $R\mid n$
while $V\mid a'$ and $(n,a')=1$. Furthermore, if there is an
admissible solution $R$ to \eqref{eq:congruence-R}, then
$(\sigma(m),V)=1$, since
\[
 (\sigma(m),V)\mid(\sigma(m)\sigma(R),V)
 =(mR,V)=(n,V)\mid(n,a')=1.
\]
Hence, each admissible solution $R$ to \eqref{eq:congruence-R} satisfies
\[
 \frac{\sigma(R)}R\equiv\frac{m}{\sigma(m)}\pmod V.
\]
If $R_1,R_2$ are two admissible solutions, taking a difference and clearing denominators shows that
\[
 \sigma(R_1)R_2-\sigma(R_2)R_1\equiv0\pmod V.
\]
But $R_1, R_2 \le x^{\varepsilon+\varepsilon^2}$, so that
\[
 |\sigma(R_1)R_2-\sigma(R_2)R_1| \le x^{2\varepsilon+2\varepsilon^2+o(1)},
 \]
while $V>x^{2\varepsilon+3\varepsilon^2}$. It follows (as in Case I) that $R_1= R_2$.

We now fix $m$ and count the possible $d,V$. If an admissible
$R$ exists, then $mx^\varepsilon\le mR=n\le X$, and so
$m\le Xx^{-\varepsilon}$. Since $b$ is a unitary divisor of
$m=Bb$ and $d\mid\sigma(b)$, we have $d\mid\sigma(m)$.
Thus, there are at most $x^{o(1)}$ choices of $d$. Moreover,
\eqref{eq:case-II-mass} gives
\[
 J=\Omega(d)\le\Omega(\sigma(b)_{>y})
 < \frac{\log x}{\log_2 x}.
\]

For fixed $d$, we estimate $\sum 1/V$ over squarefree $V\le X$
with $d\mid\sigma(V)$. The following argument is a variant of the proof of
\cite[Lemma 2.4]{Pollack2011}
(see also \cite[Lemma 2.1]{LucaPollack2011}
and \cite[Lemma 7]{Pollack2019}). We use a standard consequence of the
Brun--Titchmarsh inequality \cite[Theorem 3.9, p.~90]{MontgomeryVaughan2007}: For every positive integer $w$ and real $X\ge3$,
\begin{equation}\label{eq:prime-reciprocals}
 \sum_{\substack{\ell\le X\\\ell\equiv-1\pmod w}}\frac1\ell
 \ll\frac{\log_2 X}{\phi(w)}.
\end{equation}

Write the prime factorization of $V$ in the form
$V=\ell_1\cdots\ell_k$, so that
$\sigma(V)=(\ell_1+1)\cdots(\ell_k+1)$. Since $d\mid\sigma(V)$,
we can write
\[
 d=f_1\cdots f_k,
\]
where each $f_i\mid\ell_i+1$. We can assume, after relabeling, that
\[
 f_1\ge\cdots\ge f_h>1=f_{h+1}=\cdots=f_k.
\]
Then $f_1f_2\cdots f_h$ is a multiplicative partition of $d$.
Fix such a partition $d=f_1\cdots f_h$.
Every corresponding $V$ has a representation
$V=\ell_1\cdots\ell_hu$, with $f_i\mid\ell_i+1$.
Thus, by \eqref{eq:prime-reciprocals},
\[
\begin{aligned}
 \sum\frac1V
 &\le\Bigg\{\sum_{u\le X}\frac1u\Bigg\}
 \prod_{i=1}^h
 \Bigg\{\sum_{\substack{\ell_i\le X\\
                       \ell_i\equiv-1\pmod{f_i}}}\frac1{\ell_i}\Bigg\}\\
 &\ll\frac{(\log X)(C\log_2 X)^h}
          {\phi(f_1)\cdots\phi(f_h)}\\
 &\le\frac{(\log X)(C\log_2 X)^h\exp\{2J/y\}}d,
\end{aligned}
\]
where $C$ is a positive absolute constant. To go from the second to the third line, we used that every prime factor of $d$ exceeds $y$, so that each $\phi(f_i) = f_i \prod_{p \mid f_i}(1-1/p) \ge f_i \exp(-2\sum_{p \mid f_i} 1/p) \ge f_i \exp(-2\Omega(f_i)/y)$.

We have $h\le\Omega(d)=J$, and the number of multiplicative
partitions of $d$ is at most $\Omega(d)^{\Omega(d)}=J^J$: Each partition is obtained in at least one way by partitioning the
$\Omega(d)$ prime factors of $d$, with occurrences
distinguished, into nonempty groups. The number of these set partitions is bounded by the number $J^J$ of functions from $\{1,2,\dots,J\}$ to $\{1,2,\dots,J\}$, as one sees by encoding each partition by the function sending an element to the least member of its group. Since $J\le\log x/\log_2 x$, we have
\[
 \log\{(\log X)(C\log_2 X)^J J^J\exp\{2J/y\}\}
 \le J\log_2 x+O(J+\log_2 x).
\]
Consequently,
\begin{equation}\label{eq:reciprocal-b}
 \sum_{\substack{V\le X\\V\text{ squarefree}\\d\mid\sigma(V)}}\frac1V
 \le\frac{\exp\{J\log_2 x+O(J+\log_2 x)\}}d.
\end{equation}

By \eqref{eq:dVstuff}, all the admissible $V$ are at most
$d\,x^{\varepsilon-4\varepsilon^3}\exp\{-J\log_2 x\}$. Thus, by \eqref{eq:reciprocal-b}, their number is at most
\[ \sum_{V} \frac{d\,x^{\varepsilon-4\varepsilon^3}\exp\{-J\log_2 x\}}{V} \le x^{\varepsilon-4\varepsilon^3+o(1)}.
\]
Here the sum is over admissible \(V\), and we have used \(J+\log_2 x=o(\log x)\).

Each $m,V$ gives at most one admissible $R$, and hence at most
one $n$. For fixed $m,d$, we have just seen that the number of corresponding $V$ is bounded by $x^{\varepsilon-4\varepsilon^3+o(1)}$.
Thus, it remains to sum over $m\le Xx^{-\varepsilon}$ and
$d\mid\sigma(m)$. For each $m$, there are at most
$\tau(\sigma(m))=x^{o(1)}$ choices of $d$, uniformly in this range.
Consequently, the number of nonexceptional $n$ in Case II is at most
\[
 Xx^{-\varepsilon}\cdot x^{o(1)}
 \cdot x^{\varepsilon-4\varepsilon^3+o(1)}
 =x^{1-4\varepsilon^3+o(1)},
\]
which is negligible.

Finally, each of our amicable pairs contributes at most two
amicable numbers to $A(x)$, so
$A(x)\le xL(x)^{-1/2+4\varepsilon+o(1)}$.
Since this holds for every fixed $0<\varepsilon<1/10$,
the theorem follows.
\end{proof}


\begin{thebibliography}{99}
\providecommand{\bysame}{\leavevmode\hbox to3em{\hrulefill}\thinspace}

\bibitem{ChernykhDatabase}
S. Chernykh,
\href{https://sech.me/ap/}{\emph{Amicable numbers list}},
online database, \url{https://sech.me/ap/},
accessed September 29, 2026.

\bibitem{Erdos1935}
P. Erd\H{o}s,
\href{https://www.renyi.hu/~p_erdos/1935-08.pdf}{\emph{On the normal number of prime factors of $p-1$ and some related problems concerning Euler's $\phi$-function}},
Quart. J. Math. Oxford Ser. \textbf{6} (1935), 205--213.

\bibitem{Erdos1955}
\bysame,
\href{https://www.renyi.hu/~p_erdos/1955-03.pdf}{\emph{On amicable numbers}},
Publ. Math. Debrecen \textbf{4} (1955), 108--111.

\bibitem{ErdosRieger1975}
P. Erd\H{o}s and G.~J. Rieger,
\href{https://www.renyi.hu/~p_erdos/1975-10.pdf}{\emph{Ein Nachtrag \"uber befreundete Zahlen}},
J. reine angew. Math. \textbf{273} (1975), 220.

\bibitem{ErdosSzekeres1934}
P. Erd\H{o}s and G. Szekeres,
\href{https://www.renyi.hu/~p_erdos/1934-05.pdf}{\emph{\"Uber die Anzahl der Abelschen Gruppen gegebener Ordnung und \"uber ein verwandtes zahlentheoretisches Problem}},
Acta Sci. Math. (Szeged) \textbf{7} (1934), 95--102.

\bibitem{HardyWright2008}
G.~H. Hardy and E.~M. Wright,
\href{https://academic.oup.com/book/54489}{\emph{An introduction to the theory of numbers}},
6th ed., Oxford University Press, Oxford, 2008.

\bibitem{Iamblichus1894}
Iamblichus,
\href{https://books.google.com/books?id=ItofAAAAMAAJ}{\emph{In Nicomachi arithmeticam introductionem liber}},
E. Pistelli (ed.), B.~G. Teubner, Leipzig, 1894.

\bibitem{Kanold1954}
H.-J. Kanold,
\href{https://doi.org/10.1007/BF01181341}{\emph{\"Uber die Dichten der Mengen der vollkommenen und der befreundeten Zahlen}},
Math. Z. \textbf{61} (1954), 180--185.

\bibitem{LucaPollack2011}
F. Luca and P. Pollack,
\href{https://doi.org/10.5802/jtnb.783}{\emph{An arithmetic function arising from Carmichael's conjecture}},
J. Th\'eor. Nombres Bordeaux \textbf{23} (2011), 697--714.

\bibitem{MontgomeryVaughan2007}
H.~L. Montgomery and R.~C. Vaughan,
\href{https://websites.umich.edu/~hlm/mnt1.html}{\emph{Multiplicative number theory I. Classical theory}},
Cambridge Studies in Advanced Mathematics, vol.~97,
Cambridge University Press, Cambridge, 2007.

\bibitem{Pollack2011}
P. Pollack,
\href{https://www.pollack-math.net/combined3.pdf}{\emph{On the greatest common divisor of a number and its sum of divisors}},
Michigan Math. J. \textbf{60} (2011), 199--214.

\bibitem{Pollack2019}
\bysame,
\href{https://www.pollack-math.net/phicollapse8.pdf}{\emph{Popular subsets for Euler's $\phi$-function}},
Math. Ann. \textbf{374} (2019), 253--271.

\bibitem{Pollack2021}
\bysame,
\href{https://doi.org/10.1007/s00209-021-02762-2}
{\emph{The distribution of numbers with many factorizations}},
Math. Z. \textbf{299} (2021), 2327--2339.

\bibitem{PollackSinghaRoy2022}
P. Pollack and A. Singha Roy,
\href{https://doi.org/10.4064/cm8616-10-2021}{\emph{Powerfree sums of proper divisors}},
Colloq. Math. \textbf{168} (2022), 287--295.

\bibitem{Pomerance1977}
C. Pomerance,
\href{https://math.dartmouth.edu/~carlp/Amicable1.pdf}{\emph{On the distribution of amicable numbers}},
J. reine angew. Math. \textbf{293/294} (1977), 217--222.

\bibitem{Pomerance1980}
\bysame,
\href{https://math.dartmouth.edu/~carlp/popular.pdf}
{\emph{Popular values of Euler's function}},
Mathematika \textbf{27} (1980), 84--89.

\bibitem{Pomerance1981}
\bysame,
\href{https://math.dartmouth.edu/~carlp/Amicable2.pdf}{\emph{On the distribution of amicable numbers. II}},
J. reine angew. Math. \textbf{325} (1981), 183--188.

\bibitem{Pomerance2015}
\bysame,
\href{https://math.dartmouth.edu/~carlp/amicablesv3.pdf}{\emph{On amicable numbers}},
Analytic Number Theory, Springer, Cham, 2015, pp.~321--327.

\bibitem{Rieger1973}
G.~J. Rieger,
\href{https://doi.org/10.1515/crll.1973.261.157}{\emph{Bemerkung zu einem Ergebnis von Erd\H{o}s \"uber befreundete Zahlen}},
J. reine angew. Math. \textbf{261} (1973), 157--163.

\end{thebibliography}
    \end{document}